\documentclass[11pt]{article}
\usepackage{color,latexsym,amsthm,amsmath,amssymb,path,enumitem,url}

\newcommand{\ignore}[1]{}

\newtheorem{theorem}{Theorem}[section]
\newtheorem{lemma}[theorem]{Lemma}
\newtheorem{corollary}[theorem]{Corollary}

\newcommand{\Proof}[1]
        {
        \noindent
        \emph{Proof #1.}~
        }
\newsavebox{\smallProofsym}                     
\savebox{\smallProofsym}                        %
        {
        \begin{picture}(7,7)                    %
        \put(0,0){\framebox(6,6){}}             %
        \put(0,2){\framebox(4,4){}}             %
        \end{picture}                           %
        }                                       %
\newcommand{\smalleop}[1]
        {
        \mbox{} \hfill #1~~\usebox{\smallProofsym}\!\!\!\!\!\!\
        }

\newcommand{\parag}[1]{\vspace{2mm}

\noindent{\bf #1} }

\usepackage{graphicx,psfrag}
\usepackage[rflt]{floatflt}

\usepackage{dsfont}

\begin{document}
\pagenumbering{arabic}

\title{Improved Bounds for Progression-Free Sets in $C_{8}^{n}$}
\date{}

\author{
Fedor Petrov\thanks{St. Petersburg State University, St. Petersburg, Russia; e-mail: {\sl f.v.petrov@spbu.ru}. Research supported by Russian Science Foundation grant 17-71-20153.}
\and
Cosmin Pohoata\thanks{California Institute of Technology, Pasadena, CA, USA; e-mail: {\sl apohoata@caltech.edu}.}}

\maketitle

\begin{abstract}
Let $G$ be a finite group, and let $r_{3}(G)$ represent the size of the largest subset of $G$ without non-trivial three-term progressions. In a recent breakthrough, Croot, Lev and Pach proved that $r_{3}(C_{4}^{n}) \leqslant (3.61)^{n}$, where $C_{m}$ denotes the cyclic group of order $m$. For finite abelian groups $G \cong \prod_{i=1}^{n} C_{m_{i}}$, where $m_{1},\ldots,m_{n}$ denote positive integers such that $m_{1} | \ldots | m_{n}$, this also yields a bound of the form $r_{3}(G) \leqslant (0.903)^{\operatorname{rk}_{4}(G)} |G|$, with $\operatorname{rk}_{4}(G)$ representing the number of indices $i \in \left\{1,\ldots,n\right\}$ with $4\ |\ m_{i}$. In particular, $r_{3}(C_{8}^{n}) \leqslant (7.22)^{n}$. In this paper, we provide an exponential improvement for this bound, namely $r_{3}(C_{8}^{n}) \leq (7.09)^{n}$.
\end{abstract}

\def\codim{{\rm{codim}}\,}

\section{Introduction}
Let $G$ be a finite group. A non-trivial
three-term progression in $G$ is an ordered triple $(a,b,c)\in G^3$ of {\it{mutually distinct}}
elements such that $ac=b^2$. Let $r_{3}(G)$ be the size of the largest $A\subset G$ without non-trivial three-term progressions. 

The problem of upper bounding $r_{3}(C_{n})$ has a long history, the first important estimate being established by Roth in \cite{Roth}. Currently the 
best known upper bound is due to Bloom \cite{Bloom}, who proved that
$$r_{3}(C_n)\ll \frac{(\log\log n)^4}{\log n}n.$$
The best known lower bound is of the form
$$r_{3}(C_n) \gg n \exp(-c\sqrt{\log n})$$
for some absolute constant $c>0$ and is due to Behrend \cite{Behrend}. In particular, $r_{3}(C_n)$ grows faster than
$n^{1-\epsilon}$ for any fixed $\epsilon>0$. 

For other groups $G$, $r_{3}(G)$ turns out to be much smaller than $|G|$. The first result of this kind was obtained by Croot, Lev and Pach in their recent breakthrough paper \cite{CLP}, where they showed that 
$$r_{3}(C_4^n)\leqslant 4^{\gamma n} \approx (3.61)^{n}.$$
The constant $\gamma$ in their paper is given by
\begin{equation} \label{entropy}
\gamma:=\max\left\{\frac{1}{2}\left(\mathcal{H}_{2}(0.5-\epsilon)+\mathcal{H}_{2}(2\epsilon)\right):\ 0 < \epsilon < 0.25\right\} \approx 0.926,
\end{equation}
where $\mathcal{H}_{2}(\theta)$ denotes the binary entropy function
$$\mathcal{H}_{2}(\theta) = -\theta \log_2\theta - (1-\theta)\log_2(1-\theta),\ \ \theta \in (0,1).$$
This constant arises naturally in their polynomial method proof, which makes clever use of the group structure of $C_{4}^{n}$. This was a remarkable improvement on the previous known bounds for $G = C_{4}^{n}$, the prior record due to Sanders \cite{Sanders} being of the form
$$r_{3}(C_{4}^{n}) \ll \frac{4^{n}}{n(\log n)^{\epsilon}}$$
with an absolute constant $\epsilon > 0$. Soon after, their method was adapted and simplified in setups with more pleasant group structure. First, Ellenberg and Gijswijt in \cite{EG} proved that
$r_{3}(C_{p}^n)\leqslant \kappa_p^n$ for all odd primes $p$, where $\kappa_{n}$ generally stands for
\begin{equation}\label{kappa}
\kappa_n:=\min\left\{x^{(1-n)/3}(1+x+\dots+x^{n-1}): x>0)\right\}.
\end{equation}
This was another major result, as it improved dramatically the celebrated estimate
$$r_{3}(C_{3}^{n}) \ll \frac{3^{n}}{n^{1+\epsilon}}$$
of Bateman and Katz \cite{BK}. This was further adapted by three different teams to prove that for all odd prime powers $q$, $r_{3}(C_{q}^n)\leqslant \kappa_{q}^n$ 
(\cite{BCCGNSU}, \cite{DSB}, and \cite{Grouprings}), and also later on by various other authors to prove several other different results in extremal combinatorics.

The group algebra approach from \cite{Grouprings} allows one to estimate $r_{3}(G)$ for groups which are not necessary abelian. Nonetheless, 
all such extensions have been so far about groups of odd order. One of the difficulties about groups of even order consists of the fact that they may contain ``semi-trivial'' progressions
$(a,b,a)$ with $a^2=b^2$ and $a \neq b$. In particular, an estimate
for the number of so called {\it{multiplicative matchings}}\footnote{Multiplicative matchings coincide with what initially were called {\it{tricolored sum-free sets}} in \cite{KSS}; the updated term is adopted from Aaronson \cite{Aaronson} and Sawin \cite{Sawin}.} is no longer an estimate for $r_{3}(G)$. The first aim of this paper is to give a group algebra proof of the fact that $r_{3}(C_4^n)\leqslant (3.61)^{n}$, with a more motivated account for the constant $\kappa_{4} \approx 3.61$. The purpose of this is two-fold. First, it will reconcile the expression from \eqref{entropy} with the one from \eqref{kappa}, thus showing a clear analogy between $C_{4}^{n}$ and the odd prime power regime. Second, it will provide a framework that will allow us to give improved bounds for progression-free sets in other (abelian) $2$-groups, which is the main goal of our paper. 

For a finite abelian group $G \cong \prod_{i=1}^{n} C_{m_{i}}$ with positive integer $m_{1} | \ldots | m_{n}$, denote by $\operatorname{rk}_{4}(G)$ the number of indices $i \in \left\{1,\ldots,n\right\}$ with $4 | m_{i}$. Since $G$ is a union of $4^{-\operatorname{rk}_{4}(G)}|G|$ cosets of a subgroup isomorphic to $C_{4}^{\operatorname{rk}_4(G)}$, this yields a bound of the form 
\begin{equation} \label{CLP2}
r_{3}(G) \leqslant 4^{-(1-\gamma)\operatorname{rk}_{4}(G)}|G| \approx (0.903)^{\operatorname{rk}_{4}(G)} |G|.
\end{equation}
This is the content of Corollary 1 in \cite{CLP}. For instance, if $G = C_{8}^{n}$, the above gives 
$$r_{3}(C_8^n)\leqslant 2^n \cdot r_{3}(C_4^n) \leq (7.22)^{n}.$$
In Section \ref{c8}, we improve on this estimate and show the following
\begin{theorem} \label{8}
If $A \subset C_{8}^{n}$ is a set without non-trivial three-term progressions, then 
$$|A| \leqslant \left(2 \cdot 2^{\mathcal{H}_{4}(\rho_{0})}\right)^n \approx (7.0899)^n,$$
where $2^{\mathcal{H}_{4}(\rho)}$ represents a weighted version of $\kappa_{4}$ given by
$$2^{\mathcal{H}_{4}(\rho)} = \min_{x>0} \left\{x^{-3\rho}(1+x+x^2+x^3)\right\},$$
and $\rho_{0} \approx 0.32$ solves the system
$$\mathcal{H}_4(\rho)=\mathcal{H}_2(\theta_1)+\mathcal{H}_2(1-2\theta_1),\ \mathcal{H}_4(1-2\rho)=1+\mathcal{H}_2(1-2\theta_1)$$
for $\theta_{1} \in [x_{0},1]$ and $\rho \in [1/4,1/2]$. Here, the constant $x_{0}$ stands for the unique maximum point of the function $\mathcal{H}_2 (1-2x)+\mathcal{H}_2(x)$ in $[1/4,1/2]$.
In particular,
$$r_{3}(C_{8}^{n}) \leqslant (7.09)^{n}.$$
\end{theorem}

\bigskip

For finite abelian groups, it is also worth mentioning the following consequence.

\begin{corollary} If a finite abelian group $G$ is written as 
$$G \cong \prod_{i=1}^{n} C_{m_{i}},$$
where $m_{1} | \ldots | m_{n}$, then
$$r_{3}(G) \leqslant (0.886)^{\operatorname{rk}_{8}(G)} |G|,$$
where $\operatorname{rk}_{8}(G)$ denotes the number of indices $i \in \left\{1,\ldots,n\right\}$ with $8 | m_{i}$.
\end{corollary}

This is of course similar in spirit with Corollary 1 from \cite{CLP} and constitutes an improvement in various other cases beyond Theorem \ref{8}. Like before, it follows immediately from Theorem \ref{8} due to the simple fact that if $n:=\operatorname{rk}_{8}(G)$, then the group $G$ is a union of $8^{-n}|G|$ cosets of a subgroup isomorphic to $C_{8}^n$.

\bigskip

\section{Regularization and Tensor Power Trick}

Before we begin, we will first prove a couple of lemmas which will allow us to reduce the problem of upper bounding the size of the largest subset of $C_{4}^{n}$ without non-trivial three-term progressions to upper bounding the size of the largest three-term progression-free subset of $C_{4}^{n}$ which has the further property that it roughly intersects each of the $2^{n}$ cosets of $C_{2}^{n}$ in the same number of elements. 

Let $\Omega$ be a finite set which is partitioned into classes of size at most $m$. A subset $A\subset \Omega$ is called \emph{regular} if there exists an integer $k$ such that $|A\cap \mathcal{C}|\in \{0,k\}$ for every class $\mathcal{C}$. Suppose further that each element $x\in A$ has a non-negative \textit{weight}
$w(x)$, and define the the weight of a subset $B\subset A$ by
$$w(B):=\sum_{x\in B} w(x).$$

\begin{lemma}\label{regular} 
Every set $A\subset \Omega$ contains
a regular subset of weight at least $w(A)/H_m$, where $H_m=1+1/2+\dots+1/m$.
\end{lemma}

\begin{proof} Assume the contrary: $A$ does not 
contain such a regular subset. 
For each $i\in \{1,2,\dots,m\}$, each class $\mathcal{C}$ with at least $i$ elements of $A$ contains
a subset of $A$ of size $i$ and weight at least $\frac{i}{|\mathcal{C}\cap A|}w(\mathcal{C}\cap A)$.
Thus by our assumption
$$
i \cdot \sum_{\mathcal{C}:|\mathcal{C}\cap A|\geqslant i} \frac{w(\mathcal{C}\cap A)}{|\mathcal{C}\cap A|} < \frac{w(A)}{H_{m}}.
$$
Divide this inequality 
by $i$ and sum up over all $i=1,2,\dots,m$. We get $w(A)<w(A)$,
a contradiction.
\end{proof}

Now assume that the universe $\Omega$ is partitioned into
classes of size at most $m$, and the set of
classes is subsequently partitioned into super-classes, each class consisting of at most $m'$ super-classes. For example, classes may correspond to residues modulo
$100$ and super-classes to residues modulo $10$. A subset
$A\subset \Omega$ is called \emph{super-regular}, if 
there exist integers
$k,k'$ such that for every class $\mathcal{C}$, we have $|A\cap \mathcal{C}|\in \{0,k\}$ and the restriction $A \cap \mathcal{C}$ consists of either $0$ or $k'$ super-classes.

\begin{lemma}\label{superregular} 
In the above setting, any set $A\subset \Omega$ contains
a super-regular subset of weight at least $w(A)/H_m H_{m'}$.
\end{lemma}

\begin{proof}
By Lemma \ref{regular}, we find a regular subset
$B\subset A$ (with respect to the partition into classes) of weight at least $w(A)/H_m$. Consider the classes
which have non-empty intersection of $B$; their weights are well-defined, so the conclusion follows by applying again Lemma \ref{regular}
to the partition of these classes into superclasses.
\end{proof}

A similar statement holds for the higher hierarchy of
partitions, and is proved in the same way. 

Throughout the paper we apply both Lemmas \ref{regular}, \ref{superregular}
for the weight function equal to $1$ everywhere. For the group $G=C_4^n$, we consider the subgroup generated by its involutions, i.e. the image and the kernel of the endomorphism of $C_{4}^{n}$ defined by $g \mapsto g^{2}$; this is a copy of $C_{2}^{n}$, so we can partition $C_{4}^{n}$ into $2^n$ cosets modulo the subgroup $C_2^n=\{g^2:g\in G\}$. Thus by Lemma \ref{regular} every subset $A\subset C_4^n$ contains 
a regular subset $B$ of size at least $|A|/H_{2^n}$. For the group $C_8^n$, define
the classes and superclasses as equivalence classes of the relations
$$g\sim h\ \text{if}\ g^2=h^2\ \ \text{and}\ \ g\sim h\ \text{if}\ g^4=h^4,$$
respectively. Then by Lemma \ref{superregular} every subset 
$A\subset C_8^n$ contains 
a super-regular subset $B$ of size at least $|A|\cdot (H_{2^n})^{-2}$.

Returning to sets without three-term progressions, note that for arbitrary groups $G_{1}$, $G_{2}$, the product of two such sets $A_1\subset G_1$, $A_2\subset G_2$ is itself a subset in $G_1\times G_2$
without three-term progressions. Hence 
$$r_3(G_1\times G_2)\geqslant r_3(G_1)r_3(G_2).$$
In particular, by Fekete's Lemma on subadditive sequences \cite{MF},
\begin{equation}
\lim_{n\to \infty}(r_{3}(G^n))^{1/n}=\sup_{n > 0}\ (r_{3}(G^n))^{1/n}.
\end{equation}
This implies that any estimate of the form
$r_{3}(G^n)\leqslant c^{n+o(n)}$ automatically
yields $r_{3}(G^n)\leqslant c^n$. In particular, Lemmas 2.1 and 2.2 above reduce the problem of proving subexponential upper bounds 
$c^n$ for the size of the largest subset of $C_{4}^{n}$ 
or $C_8^n$ without three-term progressions to proving that regular (respectively, super-regular) three-term progression-free subsets of 
the group $C_4^n$ (respectively, $C_8^n$) have size $\leqslant c^{n+o(n)}$.


\section{Subspaces with zero product in abelian 2-groups}

In this section, we build the general framework that we will use for the proof of Theorem \ref{8}. Along the way, we explain the natural relationship between 
$$\frac{1}{2} \cdot \max_{0 < \epsilon < 0.25}\left\{\mathcal{H}(0.5-\epsilon)+\mathcal{H}(2\epsilon)\right\} \approx 0.926$$
and 
$$\kappa_4:=\min_{x>0} x^{-1}(1+x+x^2+x^3)\approx 3.61.$$

Let $G=\prod_{i=1}^n C_{2^{m_i}}$ be an 
abelian $2$-group. If $X_{1},\ldots,X_{k}$ are subspaces of $\mathbb{F}_{2}[G]$, we will denote by $X_{1} \cdot \ldots \cdot X_{k}$ the product set $\left\{x_{1}\cdot\ldots \cdot x_{k}: x_{i} \in X_{i}\ \text{for every}\ i \in \left\{1,\ldots,k\right\}\right\}$. In this Section, we will be interested in subspaces whose product set equals zero. Here  $\mathbb{F}_{2}[G]$ represents the group ring of $G$ over $\mathbb{F}_{2}$, namely
$$\mathbb{F}_{2}[G]:=\mathbb{F}_2[\tau_1,\dots,\tau_n]/\langle \tau_i^{2^{m_i}}=0,i=1,\dots,n\rangle.
$$
The nilpotent elements $\tau_i$ have form $1+g_i$,
where $g_i$ are generators of the cyclic groups
$C_{2^{m_i}}$. Therefore
$\mathbb{F}_2[G]$ is linearly generated by the monomials
$\prod_{i=1}^n \tau_i^{\lambda_i}, 0\leqslant \lambda_i<2^{m_i}$.
Introducing the positive weights $w_i,i=1,\dots,n$,
we define the power of a monomial $\prod_{i=1}^n \tau_i^{\lambda_i}$
as
$\sum_{i=1}^n w_i\lambda_i$. Then 
if the sum of degrees of several monomials exceeds
$$\deg_{\max}:=\sum_{i=1}^{n} w_i(2^{m_i}-1),$$
their product equals to zero.
This allows to get quite large subspaces in $\mathbb{F}_2[G]$
with zero product. Namely, denote by $X(\theta)$ the span of all monomials of degree strictly greater than
$\theta\deg_{\max}$. Then $X(\theta_1)X(\theta_2)\cdot \ldots
X(\theta_k)=0$ provided that $\sum \theta_i\geqslant 1$. Note that $\codim X(\theta)$ equals to the number of monomials
of degree at most $\theta\deg_{\max}$. 
To estimate the number
of such monomials, we may use a Chernoff type argument, as follows.
If $0<x\leqslant 1$,
we get
$$
\codim X(\theta)\leqslant 
x^{-\theta\deg_{\max}}\prod_{i=1}^n(1+x^{w_i}+x^{2w_i}+\dots+x^{(2^{m_i}-1)w_i})=:\Phi_{\theta}(x).
$$
This may be seen from opening the brackets on the
right hand side: each monomial $\prod_{i=1}^n \tau_i^{\lambda_i}$
of degree at most $\theta\deg_{\max}$
corresponds to a contribution $x^{\sum w_i\lambda_i-\theta\deg_{\max}}\geqslant 1$. Note that if $\theta\leqslant 1/2$, we have $\Phi_{\theta}(x)\leqslant \Phi_{\theta}(1/x)$
for $x\geqslant 1$. Thus the minimum of $\Phi_{\theta}(x)$
over all positive $x$
is attained on $(0,1]$. Therefore in this case we may write
$\codim X(\theta)\leqslant \min_{x \in (0,1]} \Phi_{\theta}(x)$.

When $G=C_k^n$ for $k$ equal to some power of $2$,
we may choose the weights $w_1=w_2=\dots=w_n=1$, so this gives $$\codim X(\theta)\leqslant \left(\min_{x>0}
x^{-\theta(k-1)}(1+x+x^2+\dots+x^{k-1})\right)^n.$$
Using the notation $\mathcal{H}_k(\theta):=\log_2 
\min_{x>0} 
x^{-\theta(k-1)}(1+x+x^2+\dots+x^{k-1})$, this rewrites as
\begin{equation}\label{codimb}
\codim X(\theta)\leqslant 2^{n\mathcal{H}_k(\theta)},
\end{equation}
which we will use repeatedly throughout the paper. We note that for $k=2$
this is the usual binary entropy function
$$\mathcal{H}_2(\theta)=\min_{x>0} -\theta\log_2x+\log_2(1+x)=-\theta \log_2\theta - (1-\theta)\log_2(1-\theta).$$
We also note that with all of these notations we may rewrite \eqref{kappa} as $\log_2 \kappa_p=\mathcal{H}_p(1/3)$. 

Now, consider $\kappa_4=\min_{x>0}x^{-1}(1+x+x^2+x^3)=\min_{x>0}x^{-1}(1+x)(1+x^2)$, and let $x_0>0$ be a minimizer; that is,
$$\kappa_4=x_0^{-1}(1+x_0)(1+x_0^2)=\left(x_0^{-2\theta}(1+x_0^2)
\right)
\cdot\left(x_0^{2\theta-1}(1+x_0)\right).$$
Here $\theta\in [1/4,1/2]$ is arbitrary. It follows that
$\log_2 \kappa_4\geqslant \mathcal{H}_2(\theta)+\mathcal{H}_2(1-2\theta)$. Taking the maximum over $\theta$ we get 
\begin{equation}\label{CLPconstant}
\log_2 \kappa_4\geqslant \max_{\theta\in [1/4,1/2]} \mathcal{H}_2(\theta)+\mathcal{H}_2(1-2\theta).
\end{equation}
Actually we have an equality in \eqref{CLPconstant}.
This may be explained as follows: choose $\theta$
such that the minimum of $x^{-2\theta}(1+x^2)$ is attained
at $x_0$, this gives $\mathcal{H}_2(\theta)=\log_2 x_0^{-2\theta}(1+x_0^2)$. Then both the product 
$$
\left(x^{-2\theta}(1+x^2)
\right)
\cdot\left(x^{2\theta-1}(1+x)\right)=x^{-1}(1+x+x^2+x^3)
$$
and the first multiple have a critical point at $x_0$. Thus
so does the second multiple, and it is easy to see that it actually
attains its minimum at $x_0$. Therefore $\mathcal{H}_2(1-2\theta)=
\log_2 x_0^{2\theta-1}(1+x_0)$. Hence for this
specific value of $\theta$ we get $\mathcal{H}_2(\theta)+\mathcal{H}_2(1-2\theta)=\log_2 \kappa_4$, and the maximum over all possible values of $\theta$ is not less than $\log_2 \kappa_4$, or in other words,
\eqref{CLPconstant} is an identity. In particular,
$$\log_{2} \kappa_{4} = \max_{\theta\in [1/4,1/2]} \left\{\mathcal{H}_2(\theta)+\mathcal{H}_2(1-2\theta)\right\} = \max_{\epsilon \in (0,1/4)} \left\{\mathcal{H}_{2}(0.5-\epsilon)+\mathcal{H}_{2}(2\epsilon)\right\},$$
i.e. $\kappa_{4} = 4^{\gamma}$, where
$$\gamma = \frac{1}{2} \cdot \max_{0 < \epsilon < 0.25}\left\{\mathcal{H}(0.5-\epsilon)+\mathcal{H}(2\epsilon)\right\}.$$

\section{Croot--Lev--Pach bound for $C_4^n$ with group rings}
\label{CLP-section}

In this section, we use the subspaces with vanishing product from Section 3 to give the promised alternate proof of
$$r_{3}(C_4^{n}) \leq (3.61)^{n}.$$
For reference purposes, we state this formally one more time.

\begin{theorem} \label{C4CLP} If $A\subset C_4^n$ is a set without non-trivial three-term progressions, then 
$$|A|  \leqslant \kappa_4^n,$$
where $\kappa_{4} =\min_{x>0} x^{-1}(1+x+x^2+x^3)\approx 3.61$.
\end{theorem}

\begin{proof} From the regularization argument (Lemma \ref{regular}) and tensor power trick from Section 2, 
it is enough to prove $|A|\leqslant 2\kappa_4^n$ holds whenever $A$ is a regular subset of $C_{4}^{n}$. To do this, we will proceed by contradiction. Assume that $|A|>2\kappa_4^n$, and let $\alpha > 0$ and $\beta\in [1,2]$ be such that $\kappa_{4}=\alpha \cdot \beta$ and suppose that there are $\beta^n$ classes modulo $C_2^n$ present in $A$ with the property that each class contains more than $2\alpha^n$ elements of $A$. 

We would like to emphasize at this early point that if $a \in A$ belongs to such a class $g_0\cdot C_2^n$, then $g_{0}^{2} = a^{2}$, so $\beta^{n} = |A^{2}|$, where $A^{2}$ denotes the set $\left\{x^{2}:\ x \in A\right\}$. Next, choose $\theta\in [1/4,1/2]$ 
such that $\log_2 \beta=\mathcal{H}_2(1-2\theta)$. Consider the subspaces $X(1-2\theta),X(\theta)$ in $\mathbb{F}_2[C_2^n]$
defined in Section 3. By \eqref{codimb},
$$\codim X(1-2\theta)\leqslant 2^{n \mathcal{H}_2(1-2\theta)}=\beta^n=|A^2|,$$
$X(1-2\theta)$ must have a common non-zero element with the subspace of $\mathbb{F}_{2}$-valued functions supported on $A^{-2}=\{h^{-1}\ |\ h\in A^{2}\}$.
In other words, 
there exists a non-zero element of the form 
\begin{equation}\label{3}
\sum_{h\in A^2} \eta(h^{-1})h^{-1}\in X(1-2\theta).
\end{equation}
Fix $g_0\in A$ such that $\eta(g_0^{-2})\ne 0$, and let $C=g_0\cdot C_2^n\cap A$; by our assumption on $A$, we know that $|C|>2\alpha^n$.

Consider the product
\begin{equation}\label{PLM}
\left(\sum_{g\in C} \varphi(g)g_{0}^{-1} g\right)\left(\sum_{g\in C} \psi(g) g_{0}^{-1} g\right)\left(\sum_{h\in A^2} \eta(h^{-1})h^{-1}\right)
\end{equation}
inside the group algebra $\mathbb{F}_2[C_4^n]$, where the functions $\varphi,\psi\ :\ C \to \mathbb{F}_{2}$ are chosen so that 
\begin{equation}\label{1}
\sum_{g\in C} \varphi(g) g_0^{-1}g,\ \sum_{g\in C} \psi(g) g_0^{-1}g\in X(\theta). 
\end{equation}

This product equals to $0$, since $X(\theta)X(\theta)X(1-2\theta)=0$. On the other hand, $A$ does not contain non-trivial three-term progressions, so the
coefficient of $g_{0}^{-2}$ in this product also equals $\sum_{g\in C}\varphi(g)\psi(g)\eta(g_0^{-2})$. Together with $\eta(g_{0}^{-2}) \neq 0$, this yields
$$\sum_{g\in C}\varphi(g)\psi(g)=0$$
for every $\varphi,\psi$ satisfying \eqref{1}. However, the vector subspace of $\mathbb{F}_{2}^{C}$ spanned by the functions $\varphi$ with $\sum_{g\in C} \varphi(g) g_0^{-1}g \in X(\theta)$ has codimension at most $\codim X(\theta)$, and so does the subspace spanned by the functions $\psi$ such that $\sum_{g\in C} \psi(g) g_0^{-1}g\in X(\theta)$. By \eqref{codimb}, the sum of their codimensions is at most $2 \cdot \codim X(\theta)\leqslant 2 \cdot 2^{n\mathcal{H}_2(\theta)}$, while $\log_2 \beta=\mathcal{H}_2(1-2\theta)$, which by \eqref{CLPconstant} yields $\log_2 \alpha\geqslant \mathcal{H}_2(\theta)$. Putting these together, we conclude that the sum of the codimensions of these spaces is at most 
$$2 \cdot \codim X(\theta)\leqslant 2 \cdot 2^{n\mathcal{H}_2(\theta)} \leqslant 2\alpha^n < |C|,$$
which is a contradiction, since this means the subspaces can't be orthogonal with respect to the bilinear form $\sum_{g\in C} \varphi(g)\psi(g)$.
\end{proof}

\bigskip

For the proof of Theorem \ref{8}, we will require a few additional tools.

\section{Improved bounds for progression-free sets in $C_{8}^{n}$} \label{c8}

In this section, we present the proof of Theorem \ref{8}. We begin with some further linear algebraic preliminaries.

\def\supp{{\rm supp}\,}

\def\codim{{\rm codim}\,}

\begin{lemma} \label{la0} If $X_1,X_2$ are the subspaces of 
a linear space $X$ over certain field, the codimension
of the subspace $X_1\cap X_2$ in $X_2$ 
does not exceed $\codim X_1$.
\end{lemma}

\begin{proof} The space $X_1$ is a set of vectors
in $X$ satisfying certain 
$m:=\codim X_1$ linear equations. The vectors
in $X_2$ satisfying these $m$ equations form a subspace of $X_2$
of codimension at most $m$. 
\end{proof}

Let $\Omega$ be a finite set, $K$ be a fixed field and $K^\Omega$ a space of $K$-valued functions on $\Omega$.
For a function $f\in K^\Omega$ denote by 
$\supp(f)=\{x\in \Omega:f(x)\ne 0\}$ the
support of $f$.

\begin{lemma} \label{la1} Suppose that $X\subset K^\Omega$
is a space of dimension $d$. Then $X$ contains 
a function $f$ with $|\supp(f)|\geqslant d$.
\end{lemma}

While simple, this observation was an important step in the Ellenberg-Gijswijt argument from \cite{EG}. We record the short proof here for the reader's convenience.

\begin{proof}
Consider $f\in X$ with maximal value of $|\supp f|$.
If $|\supp(f)|<d$, the number of equations $g(x)=0$ for $x \in \supp(f)$ is less than the dimension of $X$; in particular, there exists a non-zero function $g\in X$ which vanishes on $\supp(f)$. But then $$|\supp(f+g)|>|\supp(f)|,$$ which contradicts the choice of $f$.
\end{proof}

\bigskip

Last but not least, we will also need a generalization of a fact which we used at the end of the proof of Theorem \ref{C4CLP}.

\begin{lemma}\label{la2} Suppose that $a\in K^\Omega$ is a function for which the subspaces $X,Y\subset K^\Omega$
satisfy the condition $\sum_{x\in \Omega} a(x)f(x)g(x)=0$
for all $f\in X,g\in Y$. Then, 
$$\codim X+\codim Y\geqslant |\supp(a)|.$$
\end{lemma}

\begin{proof}
Assume the contrary. Denote $\Omega_0=\supp(a)$.
There is a natural embedding of $K^{\Omega_0}$
into $K^\Omega$. By Lemma \ref{la0}, the subspaces $X_0=X\cap K^{\Omega_0}$,
$Y_0=Y\cap K^{\Omega_0}$ have codimensions in $K^{\Omega_0}$
at most
$\codim X,\codim Y$ respectively. But they are orthogonal
subspaces with respect to the full rank bilinear form
$$\langle f,g\rangle:=\sum_{x\in \Omega_0} a(x)f(x)g(x).$$
Thus the sum of their codimensions is at least $|\Omega_0|=|\supp(a)|$,
and the statement of Lemma \ref{la2} is proved.
\end{proof}

\def\span{{\rm span}\,}

\parag{Using the subgroup generated by squares.}We move on to showing a general lemma about progression-free sets in finite groups, which is the key to our arguments and which may be of independent interest.

Let $G$ be a finite group, and let $H=\{g^2:g\in G\}$. We assume that $H$ is a subgroup of $G$
(in particular, this is so
in the abelian case, or for the groups of odd order,
when simply $H=G$). In this case, $H$ is a normal
subgroup due to the identity $hg^2h^{-1}=(hgh^{-1})^2$. Furthermore, fix an arbitrary field $K$. For a subset $A\subset G$, we identify $A^K$
with a span of $A$ as a subset
of the group algebra $K[G]$. In particular, we have that $H^K=K[H]$. 

\begin{lemma}\label{main-lemma}
Let $X,Y,Z$ be subspaces of $K[H]$ which satisfy $XYZ=0$. Suppose $A \subset G$ satisfies the following conditions:

(i) $|A^2| \geq 5 \cdot \codim Y$;

(ii) all elements of $A^2$ have the same number of square roots
in $A$;

(iii) each $H$-coset contains either no elements 
of $A$ or more than $\frac54(\codim X+\codim Z)$ elements of $A$.

Then, $A$ contains a three-term progression.
\end{lemma}

\begin{proof}
Suppose that $A$ does not contain three-term progressions. First, note that $A^2\subset H$. If $A^{-2}$ once again denotes the set $\left\{h^{-1}\ |\ h \in A^{2}\right\}$, fix a function $\eta : A^{-2} \to K$ such that $\sum_{c\in A^{-2}} \eta(c)c$ belongs to $Y$ and with the property that
$$|\supp(\eta)|\geqslant |A^{-2}|-\codim Y\geqslant \frac45 |A^{-2}|.$$ Such a map $\eta$ exists by Lemma \ref{la1}. In the second inequality, we made use of condition (i). For convenience, let $y_{0}:=\sum_{c\in A^{-2}} \eta(c)c$. Furthermore, consider an arbitrary coset $g_0H=Hg_0$ and choose two arbitrary functions
$\varphi: A \cap g_{0}H \to K$, $\psi(h): A \cap g_{0}H \to K$ such that 
$$x:=\sum_{a\in g_0H} \varphi(a)g_0^{-1}a\in X\ \ \text{and}\ \ z:=\sum_{b\in Hg_0} \psi(b)bg_0^{-1}\in Z.$$
Since $XYZ=0$, we have that $xy_0z=0$, so the coefficient of 
$g_0^{-2}$ in this product equals $0$. On the other hand, it equals 
$$\sum_{a\in g_0H\cap A,b\in g_0H\cap A,c\in A^{-2}:acb=1}
\varphi(a)\eta(c)\psi(b).
$$
However, $A$ does not contain three-term progressions, so $acb=1$ implies that
$a=b,c=a^{-2}$. In particular, we get that
\begin{equation} \label{bilinear}
\sum_{a\in g_0H\cap A} \varphi(a)\psi(a)\eta(a^{-2})=0.
\end{equation}

We claim that for a certain $g_{0} \in G$, this is a contradiction with Lemma \ref{la2}. To see this, recall first that the choice of $\eta$ assured us that at least $\frac45|A|$ elements $a\in A$ are such that $a^{-2}\in \supp(\eta)$. By the pigeonhole
principle, this means that there exists a coset $g_0H$ such that at least
$\frac45|A\cap g_0H|$ elements of $A\cap g_0H$ satisfy this condition. 
Since the vector space spanned by the functions $\Psi \in K^{A \cap g_{0}H}$ such that $\sum_{a\in g_0H} \Psi(a)g_0^{-1}a\in X$ (respectively, such that $ \sum_{b\in Hg_0} \Psi(b)bg_0^{-1}\in Z$) has codimension at most $\codim X$ (respectively, $\codim Z$), Lemma 5.1 and \eqref{bilinear} imply that
$$\codim X+\codim Z\geqslant \left| \left\{a \in A \cap g_{0}H\ :\ a^{-2} \in \supp(\eta)\right\} \right| \geqslant \frac45|A\cap g_0H|,$$
a contradiction with (iii).
\end{proof}

\bigskip

We will use this lemma to first complete the proof of Theorem \ref{8}.

\bigskip

{\it{Proof of Theorem \ref{8}}}. From the regularization argument (Lemma \ref{superregular}) and the tensor power trick from Section 2, it is enough to prove $|A|\leqslant 10.05 \cdot (7.09)^{n}$ in the case when $A$ is super-regular subset of $C_{8}^{n}$ without three-term progressions. Accordingly, suppose that $A$ is covered by $4.02 \cdot \gamma^n$
classes modulo $C_4^n=G^2$, where each such class contains itself $\frac{5}{4.02} \cdot \beta^n$ classes modulo $C_2^n=G^4$, with the property that each subclass modulo $C_{2}^{n}$ intersects $A$ in precisely $2.01 \cdot \alpha^n$ elements. In particular, $|A|=10.05 \cdot (\alpha \beta \gamma)^n,|A^2|=5 \cdot (\gamma \beta)^n,
|A^4|=4.02 \cdot \gamma^n$. In this setup, note that we may also assume that $\alpha$, $\beta$, $\gamma$ are all in the interval $(1,2]$. Indeed, the fact that $\alpha, \beta, \gamma \leqslant 2$ is clear since there are at most $2^{n}$ cosets of $C_{4}^{n}$ inside $C_{8}^{n}$, and at most $2^{n}$ cosets of $C_{2}^{n}$ inside $C_{4}^{n}$ (and the $C_{2}^{n}$-cosets meets $A$ in at most $2^{n}$ elements). Also, if $\min\{\alpha,\beta,\gamma\} \leqslant 1$, we get that $|A|=O(4^n)$, so we can assume from now on that $\alpha, \beta, \gamma \in (1,2]$. Furthermore, note that for each class $\mathcal{C}$ modulo $C_4^n$ which intersects $A$, we already have an upper bound for $|A \cap \mathcal{C}|$. By shifting $A \cap \mathcal{C}$ by a suitable element of $C_{8}^{n}$, we can send $A \cap \mathcal{C}$ inside the trivial coset of $C_{4}^{n}$ inside $C_{8}^{n}$. This operation preserves the property of not containing three-term progressions, so we can apply Theorem \ref{C4CLP} to write $|A \cap \mathcal{C}| \leqslant (\kappa_{4})^{n}$. The same bound also follows trivially from the super-regularity of $A$, since $A \cap \mathcal{C}$ already has the same size as the intersection of $A$ with any coset of $C_{4}^{n}$ inside $C_{8}^{n}$, however we can use the super-regular structure of $A$ more efficiently.

In light of the above, suppose without loss of generality that $A \cap \mathcal{C} \subset C_{4}^{n}$, and let $\theta \in [1/4,1/2]$ be such that $\log_2 \beta=\mathcal{H}_2(1-2\theta)$. We first claim that $\log_2 \alpha<\mathcal{H}_2(\theta)$. This follows in fact by applying the argument from Section 4 to $A \cap \mathcal{C} \subset C_{4}^{n}$. Indeed, if $\log_2 \alpha \geq \mathcal{H}_2(\theta)$ we have that 
$$\codim X(1-2\theta) \leq 2^{n\mathcal{H}_{2}(1-2\theta)} = \beta^{n}\ \ \ \text{and}\ \ \ \codim X(\theta) \leq 2^{n\mathcal{H}_{2}(\theta)} \leq \alpha^{n},$$ so we can consider once again the (zero) product from \eqref{PLM} for a suitable intersection $A_{g_{0}}$ of $A \cap \mathcal{C}$ with a coset of $C_{2}^{n}$. Similarly, the fact that $A \cap \mathcal{C}$ has no three-term progressions then produces two spaces of functions, $\Phi$ and $\Psi$, each with codimension at most $\codim X(\theta)$ in $\mathbb{F}_{2}^{A_{g_{0}}}$, which must also be orthogonal with respect to the bilinear form $\sum_{g \in A_{g_{0}}} \varphi(g) \Psi(g)$. However, the super-regularity of $A$ and Lemma 5.1 then imply
$$2.01 \cdot \alpha^{n}=|A_{g_{0}}| \leq \codim \Phi + \codim \Psi \leq 2 \cdot \codim X(\theta) \leq  2\alpha^{n},$$
which is a contradiction. Consequently, $\log_2 \alpha<\mathcal{H}_2(\theta)$, as claimed.

Next, consider $\rho\in [1/4,1/2]$ such that $\log_2 \gamma\beta=\mathcal{H}_4(1-2\rho)$. Note that $$|A^2|=5\cdot (\gamma \beta)^n = 5 \cdot 2^{n\mathcal{H}_4(1-2\rho)} \geq 5 \cdot \codim Y.$$
Applied for $\mathbb{F}_{2}[C_{4}^{n}]$ and the subspaces $Y=X(1-2\rho), X=Z=X(\rho)$, Lemma \ref{main-lemma} thus yields
$$\frac{5}{2} \cdot \alpha^{n} \beta^{n} = \left(2.01 \cdot \alpha^{n}\right)\left(\frac{5}{4.02} \cdot \beta^{n}\right) \leq \frac{5}{2} \cdot \codim X(\rho) \leq \frac{5}{2} \cdot 2^{n\mathcal{H}_{4}(\rho)},$$
which implies
$$\log_2 \alpha\beta \leqslant \mathcal{H}_4(\rho).$$

This condition imposes a special further constraint on $\alpha,\beta,\gamma$, and maximizing the product $\alpha \beta \gamma$ requires a delicate analysis which will be covered in the next subsection. For now, let us just argue
$$\log_2 \alpha\beta\gamma<c<1+\log_2\kappa_4,$$
for certain $c$, i.e. $r_{3}(C_{8}^{n}) = O(c^{n})$, where $c < 2 \kappa_{4} \approx 7.22$.  The analysis below will show roughly that if $\log_2 \alpha\beta\gamma$ is close to $1+\log_2\kappa_4$, then $\gamma$ must be close to $2$, while $\alpha\beta$ must be close to 
$2^{\kappa_4}=\mathcal{H}_4(1/3)$. Therefore 
$\log_2 \alpha\beta< \mathcal{H}_4(\rho)$ implies that 
$\rho\geqslant 1/3+o(1)$, and 
$$\log_2 2\beta=\log_2 \gamma\beta+o(1)=\mathcal{H}_4(1-2\rho)+o(1)\leqslant 
\mathcal{H}_4(1/3)+o(1)=\log_2 \alpha\beta+o(1),$$
which represents a contradiction. 

\parag{Maximizing $\alpha\beta\gamma$.}For the reader's convenience, let us first recall the restrictions we have on $\alpha$, $\beta$ and $\gamma$; in the previous subsection, we showed that there exist positive reals $\rho,\theta\in [1/4,1/2]$ 
such that 
\begin{equation}\label{cases}
\begin{cases}\log_2\beta&=\mathcal{H}_2(1-2\theta)\\
\log_2 \alpha&\leqslant \mathcal{H}_2(\theta)\\ 
\log_2 \gamma\beta&=\mathcal{H}_4(1-2\rho)\\
\log_2 \alpha\beta& \leqslant \mathcal{H}_4(\rho).
\end{cases}
\end{equation}

The maximal value of $\alpha\beta\gamma$ for $\rho,\theta\in [1/4,1/2]$,
$\alpha,\beta,\gamma\in [1,2]$
and \eqref{cases} is achieved. Denote the corresponding point $(\rho_0,\theta_0,\alpha_0,\beta_0,\gamma_0)$. Assume that
$\gamma_0<2$. If $\beta_0=1$, then we have $\alpha_0\beta_0\gamma_0\leqslant 4$,
which is definitely not a maximum, thus $\beta_0>1$. 
Choose $\gamma$ slightly greater than
$\gamma_0$ and $\beta<\beta_0$ so that $\beta_0\gamma_0=\beta\gamma$. Since the binary entropy function $\mathcal{H}_2$ is increasing on $[0,1/2]$, the new $\theta$ such that $\log_2 \beta=\mathcal{H}_2(1-2\theta)$
satisfies $\theta>\theta_0$. In particular, this means that there exists $\alpha>\alpha_0$ such that 
$\log_2\alpha\leqslant \mathcal{H}_2(\theta)$ and 
$\alpha\beta\leqslant \alpha_0\beta_0$. We have $\alpha\beta\gamma=\alpha\beta_0\gamma_0>\alpha_0\beta_0\gamma_0$,
a contradiction with maximality. Therefore the maximum is achieved for $\gamma_0=2$.
If $\alpha_0=2$, we get $\theta_0=1/2$, $\beta_0=1$ and $\alpha_0\beta_0\gamma_0=4$,
too small for a maximum. 

Next, we claim that for the point $(\rho_{0},\theta_{0},\alpha_{0},\beta_{0},2)$ which maximizes the value $\alpha\beta\gamma$ the second inequality from \eqref{cases} must be an equality. We argue this again by contradiction; suppose that $\log_2 \alpha_0<\mathcal{H}_2(\theta_0)$. 
Then we may choose
$\beta$ slightly less than
$\beta_0$, define $\rho$ by $\log_2 2\beta=\mathcal{H}_4(1-2\rho)$
and $\theta$ by $\log_2\beta=\mathcal{H}_2(1-2\theta)$.
After that we may choose $\alpha\in (\alpha_0\beta_0/\beta,2)$ 
so that \eqref{cases} still holds for $\alpha,\beta$ (and $\gamma=\gamma_0=2$), which yields a contradiction. This is indeed clear when $\log_2 \alpha_0\beta_0<\mathcal{H}_4(\rho_0)$, but even if we had equality in the last line from \eqref{cases}, namely $\log_2 \alpha_0\beta_0= \mathcal{H}_4(\rho_0)$, then we can choose $\alpha$ so that 
$$\log_2 \alpha\beta=\mathcal{H}_4(\rho)>\mathcal{H}_4(\rho_0)=\log_2\alpha_0\beta_0.$$
Therefore, $\log_2 \alpha_0=\mathcal{H}_2(\theta_0)$. We also claim that equality must hold in the last inequality from \eqref{cases}. Suppose that $\log_2 \alpha_0\beta_0< \mathcal{H}_4(\rho_0)$. The function
$\mathcal{H}_2 (1-2x)+\mathcal{H}_2(x)$ is concave on $[1/4,1/2]$, so it has an unique point of maximum, which we call $x_0$ just like in Section 3. If $\theta_0\ne x_0$, we may perturb the pair $(\alpha,\beta)$ slightly so that the product $\alpha\beta$ increases and the conditions from \eqref{cases} still hold (with
$\log_2\alpha=\mathcal{H}_2(\theta)$).  If $\theta_0=x_0$,
we have 
$$\log_2 \alpha_0\beta_0=\max_{x \in [1/4,1/2]} \left\{\mathcal{H}_2 (1-2x)+\mathcal{H}_2(x)\right\}=\mathcal{H}_4(1/3),$$ 
so $\rho_0\geqslant 1/3$, but then by the analysis from Section 3
$$\mathcal{H}_4(1/3)\geqslant \mathcal{H}_4(1-2\rho)\geqslant \log_2 \gamma_0\beta_0=\log_2 2\beta_0>\log_2 \alpha_0\beta_0=
\mathcal{H}_4(1/3),$$
which is once again a contradiction. 

We have thus proved that $\gamma_0=2$, $\log_2 \alpha_0=\mathcal{H}_2(\theta_0)$,
$\log_2 \alpha_0\beta_0=\mathcal{H}_4(\rho_0)$. Finally, let us assume that we found certain $\theta_1\in [x_0,1/2]$ and 
$\rho_1\in [1/4,1/2]$ satisfying 
\begin{equation} \label{optim}
\mathcal{H}_4(\rho_1)=\mathcal{H}_2(\theta_1)+\mathcal{H}_2(1-2\theta_1),\ \mathcal{H}_4(1-2\rho_1)=1+
\mathcal{H}_2(1-2\theta_1).
\end{equation}
We claim that $\theta_1=\theta_0,\rho_1=\rho_0$. We argue this one last time by contradiction. If $\theta_0<x_0\leqslant \theta_1$, note that we get 
$$\mathcal{H}_4(1-2\rho_1)=1+
\mathcal{H}_2(1-2\theta_1)<1+
\mathcal{H}_2(1-2\theta_0)=\mathcal{H}_4(1-2\rho_0),$$
therefore $\rho_1>\rho_0$, and we may replace $\alpha_0$ and $\beta_0$ by $\alpha$ and $\beta$ defined by $\log_2 \alpha=\mathcal{H}_2(\theta_1)$,
$\log_2\beta=\mathcal{H}_2(1-2\theta_1)$, with $\alpha\beta>\alpha_0\beta_0$, contradicting the maximality of $\alpha_{0}\beta_{0}$. If $\theta_0\geqslant x_0$,
both functions $\mathcal{H}_2(x)+
\mathcal{H}_2(1-2x)$ and $1+
\mathcal{H}_2(1-2x)$ decrease on the segment $[x_0,1/2]$
containing both $\theta_0$ and $\theta_1$. 
This implies that if, say, $\theta_0<\theta_1$,
we get $\rho_0>\rho_1$ and $1-2\rho_0>1-2\rho_1$, which is also impossible.

To pinpoint our optimizer $(\rho_0,\theta_0,\alpha_0,\beta_0,\gamma_0)$, we therefore look for
$\theta_1\in [x_0,1/2]$ and 
$\rho_1\in [1/4,1/2]$ satisfying \eqref{optim}. The first equation defines $\rho_1$ as a (strictly) decreasing
function of $\theta_1$, whereas the second equations represents it as an increasing one.
Thus such $\theta_1$ is (a priori at most) unique and the approximate estimates may be specified by
Intermediate Value Theorem. 
Numerically, the values of $\theta_1,\rho_1$ and $2^{\mathcal{H}_4(\rho_{1})+1}=2\alpha_{0}\beta_{0}$
are about $\theta_1\approx 0.343$, $\rho_1\approx 0.32$, $2\alpha_{0}\beta_{0}\approx 7.0899$. Putting everything together, we can finally conclude that
$$|A| = 10.05 \cdot (\alpha \beta \gamma)^{n} \leq 10.05 \cdot (7.0899)^{n},$$
which completes the proof of Theorem \ref{8}.

\section{Concluding Remarks}

Finding examples of large sets inside $C_{8}^{n}$ without non-trivial three-term progressions is also quite an interesting problem. As with $C_{3}^{n}$, where the best lower bound is due to Edel \cite{YE}, one would be tempted to find the largest possible three-term progression free set in $C_{8}^{k}$ for a few small values of $k$, and then output the best cartesian product construction. We believe all such attempts lead to lower bounds of the form $$r_{3}(C_{8}^{n}) = \Omega \left(c^{n}\right),$$ 
where $c < 5$. We can do better by using a Behrend-type construction. We switch to additive notation for convenience.

\begin{theorem} \label{b8}
Suppose that $G = (\mathbb{Z}/8\mathbb{Z})^{n}$. Then there is a set $A \subset G$ with no three-term progression and
$$|A| = \Omega\left(|G|^{\log 5/\log 8}/\sqrt{\log|G|}\right).$$
\end{theorem}

\begin{proof}
Consider the set $S \subset \mathbb{Z}^{n}$ consisting of the points $(x_{1},\ldots,x_{n}) \in \left\{0,1,2,3,4\right\}^{n}$ with the property that 
$$\sum_{i=1}^{n} (x_{i}-2)^{2} = 2n.$$
In other words, $S$ is the intersection of $\left\{0,1,2,3,4\right\}^{n}$ with the $n$-dimensional hypersphere centered at $(2,\ldots,2)$ and radius $n\sqrt{2}$. In particular, no three points in $S$ are collinear. Moreover, the size of $|S|$ is at $\Omega(5^{n}/\sqrt{n})$, as one can easily see from the Central Limit Theorem. Indeed, let $X$ be the random variable which takes values $0,1,2,3,4$ with probability $1/5$ each; let $X_{1},\ldots,X_{n}$ be $n$ independent copies of $X$ and let $Y_{i} = (X_{i}-2)^{2}$ for each $i=1,\ldots n$. It is easy to see that $\mathbb{E}[Y_{i}]=2$, so $|S|/5^{n}$ is the probability that that $Y_{1}+\ldots+Y_{n} = 2n$. 

Consider the identity map $\Psi\ :\{0,1,2,3,4\}^{n} \rightarrow (\mathbb{Z}/8\mathbb{Z})^{n}$ and let $A$ denote the image of $S$. We claim that $A$ does not contain non-trivial three-term $(\mathbb{Z}/8\mathbb{Z})^{n}$ arithmetic progressions. To see this, note that if $a+c = 2b$, with $a \neq c$, then either $\Psi^{-1}(a)$, $\Psi^{-1}(b)$, $\Psi^{-1}(c)$ is a three-term progression in $\mathbb{Z}^{n}$ or there must be a nonempty subset $I \subset \left\{1,\ldots,n\right\}$ such that
$$(\Psi^{-1}(a)_{i},\Psi^{-1}(b)_{i},\Psi^{-1}(c)_{i}) \in \left\{(4,0,4), (0,4,0)\right\}$$
for every $i \in I$. The former scenario is impossible, since $S$ does not contain three points in arithmetic progression. If the latter happens, we let $a',b',c' \in \left\{0,1,2,3,4\right\}^{n}$ be the points obtained from $a$, $b$, $c$ by swapping $(4,0,4)$ with $(4,4,4)$ and/or by swapping $(0,4,0)$ with $(0,0,0)$ for each coordinate $i$ where $\Psi^{-1}(a)_{i},\Psi^{-1}(b)_{i},\Psi^{-1}(c)_{i}$ is a three-term progression in $\mathbb{Z}/8\mathbb{Z}$ but not in $\mathbb{Z}$. Note that if $a$, $b$, $c$ lie on a hypersphere centered at $(2,\ldots,2)$, the points $a'$, $b'$, $c'$ must also lie on the same hypersphere. However, if $a+c=2b$ holds in $(\mathbb{Z}/8\mathbb{Z})^{n}$ then $a'+c'=2b'$ must also hold in $\mathbb{Z}^{n}$, and this is again impossible.

\end{proof}

A similar story holds for three-term progression-free sets inside $(\mathbb{Z}/4\mathbb{Z})^{n}$, where the product constructions seemingly lead only to lower bounds of the form
$$r_{3}(G) = \Omega\left(C^{n}\right),$$
where $C < 3$. One can easily adapt the above construction to get the following lower bound.
\begin{theorem}
Suppose that $G = (\mathbb{Z}/4\mathbb{Z})^{n}$. Then there is a set $A \subset G$ with no three-term progression and
$$|A| = \Omega\left(|G|^{\log 3/\log 4}/\sqrt{\log|G|}\right).$$
\end{theorem}

A similar construction of Elsholtz \cite{CE} also achieves this for $(\mathbb{Z}/4\mathbb{Z})^{n}$.

\end{document}